\documentclass[final]{dmtcs-episciences}

\usepackage[utf8]{inputenc}
\usepackage{subfigure}

\usepackage{amsmath}
\usepackage{amssymb}
\usepackage{amsthm}
\usepackage{centernot}

\usepackage[round]{natbib}

\newtheorem{theorem}{Theorem}[section]

\newtheorem{corollary}[theorem]{Corollary}

\newtheorem{lemma}[theorem]{Lemma}

\newtheorem{proposition}[theorem]{Proposition}

\graphicspath{ {./Images/} }

\newcommand{\agree}[1]{\left[#1 \right]_\Gamma}

\newcommand{\switchhom}[3]{#1\to_{#3}#2}
\newcommand{\nswitchhom}[3]{#1\centernot\to_{#3}#2}

\newcommand{\switch}[3]{#3^{(#1,#2)}}

\title{The 2-colouring problem for $(m,n)$-mixed graphs with switching is polynomial}

\author{Richard C.\ Brewster\affiliationmark{1}\thanks{Supported by NSERC (Canada).}
	\and Arnott Kidner\affiliationmark{2}
	\and Gary MacGillivray\affiliationmark{2}\thanks{Supported by NSERC (Canada).}}
	
\affiliation{Department of Mathematics and Statistics, Thompson Rivers University, Kamloops, B.C., Canada \\
Department of Mathematics and Statistics, University of Victoria, Victoria, B.C., Canada}


\received{2022-03-22}
\revised{2022-08-29} 
\accepted{2022-08-31}

\begin{document}
\publicationdetails{24}{2022}{2}{5}{9242}
\maketitle
\begin{abstract}
A mixed graph is a set of vertices together with an edge set and an arc set.		
An $(m,n)$-mixed graph $G$ is a mixed graph whose edges are each assigned one of $m$ colours, and whose arcs are each assigned one of $n$ colours.  A \emph{switch} at a vertex $v$ of $G$ permutes the edge colours, the arc colours, and the arc directions of edges and arcs incident with $v$.  The group of all allowed switches is $\Gamma$.

Let $k \geq 1$ be a fixed integer and $\Gamma$ a fixed permutation group. 
We consider the problem that takes as input an $(m,n)$-mixed graph $G$ and asks if there is a sequence of switches at vertices of $G$ with respect to $\Gamma$ so that the resulting $(m,n)$-mixed graph admits a homomorphism to an $(m,n)$-mixed graph on $k$ vertices.  Our main result establishes that this problem can be solved in polynomial time for $k \leq 2$, and is NP-hard for $k \geq 3$.  This provides a step towards a general dichotomy theorem for the $\Gamma$-switchable homomorphism decision problem.		
	\end{abstract}
	
\section{Introduction}\label{sec:intro}

Homomorphisms of graphs (and in general relational systems) are well studied generalizations of vertex colourings~\citep{Hell_Nesetril_2008}.  Given a graph (or some generalization) $G$, the question of whether $G$ admits a $k$-colouring, can be equivalently rephrased as ``does $G$ admit a homomorphism to a target on $k$ vertices?''.

In this paper we study homomorphisms of $(m,n)$-mixed graphs endowed with a switching operation under some fixed permutation group. (Formal definitions and precise statements of our results are given below.)  Our main result is that the $2$-colouring problem under these homomorphisms can be solved in polynomial time.  As $k$-colouring for classical graphs can be encoded within our framework, $k$-colouring in our setting is NP-hard for fixed $k \geq 3$.  That is, $k$-colouring for $(m,n)$-mixed graphs with a switching operation exhibits a dichotomy analogous to $k$-colouring of classical graphs~\citep{Garey_Johnson_1979}.  Thus, our work may be viewed as a first step towards a dichotomy theorem for homomorphisms of $(m,n)$-mixed graphs with a switching operation.  We remark that the $k$-colouring problem in our setting is not obviously a Constraint Satisfaction Problem~\citep{Feder_Vardi_1993, Bulatov_2017, Zhuk_2020} nor is membership in NP clear.  These ideas are explored further in a companion paper~\citep{Brewster_KM_2022}.

We begin with the key definitions to state our main result.  In this paper, all graphs and all groups are finite. 
	
A \emph{mixed graph} is a triple $G = (V(G), E(G), A(G))$ consisting of a set of vertices $V(G)$, a set of edges $E(G)$ of unordered pairs of vertices, and a set of arcs $A(G)$ of ordered pairs of vertices.  Given a pair of vertices $u$ and $v$, there is at most one edge, or one arc, but not both, joining them.  Further we assume $G$ is loop-free. We will use $uv$ to denote an edge or an arc with end points $u$ and $v$ where in the latter case the arc is oriented from $u$ to $v$.  

Mixed graphs were introduced in~\citep{Nesetril_Raspaud_2000} as an attempt to unify the theories of homomorphisms of $2$-edge coloured graphs and of oriented graphs.   Numerous similarities between the two settings have been observed (see for example, \citep{Alon_Marshall_1998, Raspaud_Sopena_1994, Kostochka_SZ_1997}), whereas, \citep{Sen_2014} provides examples highlighting key differences.

In this work we study edge and arc coloured generalizations of mixed graphs.  Thus, our work may be viewed as a unification of homomorphisms of edge-coloured graphs and of arc-coloured graphs. Let $m$ and $n$ be non-negative integers.  Denote by $[m]$ the set $\{ 1, 2, \dots, m \}$. An \emph{$(m,n)$-mixed graph} is a mixed graph $G=(V(G),E(G),A(G))$ together with functions $c: E(G) \to [m]$ and $d: A(G) \to [n]$ that assign to each edge one of $m$ colours, and to each arc one of $n$ colours respectively.  (The colour sets for edges and arcs are disjoint.)  The \emph{underlying mixed graph} of $G$ is $(V(G),E(G),A(G))$, i.e., the mixed graph obtained by ignoring edge and arc colours. The \emph{underlying graph} of $G$ is the graph obtained by ignoring edge and arc colours and arc directions.  An $(m,n)$-mixed graph is a cycle if its underlying graph is a cycle and similarly for other standard graph theoretic terms such as path, tree, bipartite, etc.
	
Fundamental to our work is the following definition. An $(m,n)$-mixed graph is \emph{monochromatic of colour $i$} if either every edge is colour $i$ and there are no arcs, or every arc is colour $i$ and there are no edges.  While a monochromatic mixed graph with only edges is naturally isomorphic to its underlying graph, we note that we still view the edges as having colour $i$.
	
Let $G$ and $H$ be $(m,n)$-mixed graphs.	A \emph{homomorphism} of $G$ to $H$ is a function $h: V(G) \to V(H)$ such that if $uv$ is an edge of colour $i$ in $G$, then $h(u)h(v)$ is an edge of colour $i$ of $H$, and if $uv$ is an arc of colour $j$ in $G$, then $h(u)h(v)$ is an arc of colour $j$ in $H$.  We denote the existence of a homomorphism of $G$ to $H$ by $G \to H$ or $h: G \to H$ when the name of the function is required.  

We now turn our attention to the concept of switching an $(m,n)$-mixed graph at a vertex $v$. This generalizes the concept of switching edge colours or signs~\citep{Brewster_Graves_2009, Zaslavsky_1982} (permuting the colour of edges incident at $v$) and pushing digraphs~\citep{Klostermeyer_MacGillivray_2004} (reversing the direction of arcs incident at $v$).
Let $\Gamma\leq S_m \times S_n \times S_2^n$ be a permutation group. An element of $\Gamma$ will act on edge colours, arc colours, and arc directions. Specifically, the element is an ordered $(n+2)$-tuple $\pi = (\alpha, \beta, \gamma_1, \gamma_2, \ldots, \gamma_n)$ where $\alpha$ acts on the edge colours, $\beta$ acts on the arc colours, and $\gamma_i$ acts on the arc direction of arcs of colour $i$.  For the remainder of the paper, $\Gamma$ will be a permutation group as described here.

Let $G$ be a $(m,n)$-mixed graph, and $\pi = (\alpha, \beta, \gamma_1, \gamma_2,$ $ \ldots, \gamma_n)\in \Gamma$. Define $\switch{v}{\pi}{G}$ as the $(m,n)$-mixed graph arising from $G$ by \emph{switching at vertex $v$ with respect to $\pi$} as follows. Replace each edge $vw$ of colour $i$ by an edge $vw$ of colour $\alpha(i)$. Replace each arc $a$ of colour $i$ incident at $v$ (i.e., $a=vx$ or $a=xv$) with an arc of colour $\beta(i)$ and orientation $\gamma_i(a)$.  Note, $\gamma_i(a) \in \{ vx, xv \}$.	

Given  a sequence of ordered pairs from $V(G) \times \Gamma$, say $\Sigma = (v_1,\pi_1)(v_2,\pi_2)\ldots$ $(v_k,\pi_k)$, we define \emph{switching $G$ with respect to the sequence $\Sigma$} as follows: 
$$
G^\Sigma =(G)^{(v_1,\pi_1)(v_2,\pi_2)\ldots(v_k,\pi_k)} =  (G^{(v_1,\pi_1)})^{(v_2,\pi_2)(v_3,\pi_3)\ldots(v_k,\pi_k)}.
$$
Note if we let $\Sigma^{-1}=(v_k,\pi_k^{-1})\dots(v_1,\pi_1^{-1})$, then $G^{\Sigma\Sigma^{-1}}=G^{\Sigma^{-1}\Sigma}=G$.

Given a subset of vertices, $X \subseteq V(G)$, we can switch at each vertex of $X$ with respect to a permutation $\pi \in \Gamma$, the result of which we denote by $\switch{X}{\pi}{G}$. This operation is well defined independently of the order in which we switch. If $uv$ is an edge or arc with one end in $X$, say $u$, then we simply switch at $u$ with respect to $\pi$. Suppose both ends of $uv$ are in $X$.  If $uv$ is an edge of colour $i$, then after switching at each vertex of $X$, the edge will have colour $\alpha^2(i)$.  If $uv$ is an arc, then after switching the colour will be $\beta^2(i)$ and the direction will be $\gamma_{\beta(i)}\gamma_i(uv)$.

	Two $(m,n)$-mixed graphs $G$ and $G'$ with the same underlying graph are \emph{$\Gamma$-switch equivalent} if there exists a sequence of switches $\Sigma$ such that $G^\Sigma = G'$.  We may simply say \emph{switch equivalent} when $\Gamma$ is clear from context. Note since $V(G) = V(G')$, we are viewing both $(m,n)$-mixed graphs as labelled and thus are not considering equivalence under switching followed by an automorphism. Such an extension of equivalence is possible but unnecessary in this work. Since $\Gamma$ is a group, the following proposition is immediate.
	
	\begin{proposition}\label{prop:switchEquivalent}
		$\Gamma$-switch equivalence is an equivalence relation on the set of (labelled) $(m,n)$-mixed graphs. 
	\end{proposition}
		
We are now ready to define switching homomorphisms.  Our definition naturally builds on push homomorphisms of digraphs~\citep{Klostermeyer_MacGillivray_2004} and homomorphisms of signed graphs, introduced by~\citep{Guenin_2005} and developed further by~\citep{Naserasr_RS_2015}.  See~\citep{Naserasr_SZ_2021} for a recent survey on signed graph homomorphisms. Let $G$ and $H$ be $(m,n)$-mixed graphs. A \emph{$\Gamma$-switchable homomorphism} of $G$ to $H$ is a sequence of switches $\Sigma$ together with a homomorphism $G^\Sigma \to H$.  We denote the existence of such a homomorphism by $\switchhom{G}{H}{\Gamma}$, 
or $f: \switchhom{G}{H}{\Gamma}$ when we wish to name the mapping.  Observe the notation $G \to H$ refers to a homomorphism of $(m,n)$-mixed graphs without switching, and $\switchhom{G}{H}{\Gamma}$ refers to switching $G$ followed by a homomorphism of (the resulting) $(m,n)$-mixed graphs.

A useful fact is the following.  If $\switchhom{G}{H}{\Gamma}$, then $\switchhom{G}{H^{(v,\pi)}}{\Gamma}$ for any $v \in V(H)$ and any $\pi \in \Gamma$.  To see this let $\Sigma$ be a sequence of switches such that $f: G^\Sigma \to H$.  Let $X = f^{-1}(v) \subseteq V(G^\Sigma)$.  It is easy to see the same vertex mapping $f: V(G) \to V(H)$ defines a homomorphism $(G^\Sigma)^{(X,\pi)} \to H^{(v,\pi)}$.  As a result of this observation, we have two immediate corollaries.  First, $\Gamma$-switchable homomorphisms compose.  Second, when studying the question ``does $G$ admit a $\Gamma$-switchable homomorphism to $H$?'' we are free to replace $H$ with any $H'$ switch equivalent to $H$.
	
For (classical) graphs, $G$ is $k$-colourable if and only if it admits a homomorphism to a graph $H$ of order $k$.  
Analogously, we say an $(m,n)$-mixed graph $G$ is \emph{$\Gamma$-switchable $k$-colourable}, if there is an $(m,n)$-mixed graph $H$ of order $k$ such that $\switchhom{G}{H}{\Gamma}$.  The corresponding decision problem is defined as follows. Let $k \geq 1$ be a fixed integer and $\Gamma\leq S_m \times S_n \times S_2^n$ be a fixed group.  We define the following decision problem.

\begin{quote}
\samepage
\begin{list}{}{}	
    \setlength{\itemsep}{0pt}

    \item[\textsc{$\Gamma$-Switchable $k$-Col}]~
        
	\item[\textsc{Input:}] An $(m,n)$-mixed graph $G$.
	
	\item[\textsc{Question:}] Is $G$ $\Gamma$-switchable $k$-colourable?
\end{list}
\end{quote}
	
Our main result is the following dichotomy result for \textsc{$\Gamma$-Switchable $k$-Col}.
	
\begin{theorem}\label{thm:main}
Let $k\geq1$ be an integer and $\Gamma\leq S_m \times S_n \times S_2^n$ be a group. If $k\leq2$, then \textsc{$\Gamma$-Switchable $k$-Col} is solvable in polynomial time. If $k\geq 3$, then \textsc{$\Gamma$-Switchable $k$-Col} is NP-hard.
\end{theorem}

The NP-hardness half of the dichotomy is immediate.

\begin{proposition}\label{prop:NPc}
For $k \geq 3$, \textsc{$\Gamma$-Switchable $k$-Col} is NP-hard.
\end{proposition}
	
\begin{proof}
Let $G$ be an instance of $k$-colouring (for classical graphs).  Let $G'$ be the $(m,n)$-mixed graph obtained from $G$ by assigning each edge colour $1$.  If $G$ is $k$-colourable, then clearly $G'$ is $k$-colourable.  (Assign all edges in $G'$ and $K_k$ the colour $1$ and use the same mapping.)  Conversely, if $G'$ is $k$-colourable, then the $\Gamma$-switchable homomorphism induces a homomorphism of the underlying graphs showing $G$ is $k$-colourable.
\end{proof}	

For an Abelian group we remark that if $G$ and $G'$ are switch equivalent, then there  is a sequence of switches $\Sigma$ of length at most $|V(G)|$ so that $G^\Sigma = G'$.  (This is discussed in more detail below.)  Thus when $\Gamma$ is Abelian, \textsc{$\Gamma$-Switchable $k$-Col} is in NP, and we can conclude for $k \geq 3$, the problem is NP-complete.  The situation for non-Abelian groups is more complicated and is studied further in~\citep{Brewster_KM_2022}.
	
It is trivial to decide if an $(m,n)$-mixed graph is $1$-colourable. Thus to complete the proof we settle the case $k=2$. Results are known when $\Gamma$ belongs to certain families of groups~\citep{Duffy_MT_2021, Leclerc_MW_2021}.  The remainder of the paper establishes the problem is polynomial time solvable for all groups $\Gamma$.
	
We conclude the introduction with a remark on the general homomorphism problem.  Let $H$ be a fixed $(m,n)$-mixed graph and $\Gamma$ a fixed permutation group. \\

\begin{quote}
\samepage
\begin{list}{}{}	
    \setlength{\itemsep}{0pt}

    \item[\textsc{$\Gamma$-Hom-$H$}]~
        
	\item[\textsc{Input:}] An $(m,n)$-mixed graph $G$.
	
	\item[\textsc{Question:}] Does $G$ admit a $\Gamma$-switchable homomorphism to $H$?
\end{list}
\end{quote}
	
The complexity of \textsc{$\Gamma$-Hom-$H$} has been investigated for the same families of groups as $\Gamma$-switchable $k$-colouring in~\citep{Duffy_MT_2021, Leclerc_MW_2021}.  The following theorem is an immediate corollary to our main result.  
We remark that similar polynomial complexity results have been proved for push homomorphisms~\citep{Klostermeyer_MacGillivray_2004} and for signed graphs~\citep{Brewster_FHN_2017}.

	\begin{theorem}\label{thm:main2}
	Let $H$ be a $2$-colourable $(m,n)$-mixed graph, then
	\textsc{$\Gamma$-Hom-$H$} is polynomial time solvable.
	\end{theorem}

\section{Restriction to $m$-edge coloured graphs}\label{sec:medge}

If a non-trivial $(m,n)$-mixed graph $G$ is $2$-colourable, then the target of order $2$ to which $G$ maps must be a monochromatic $K_2$ or a monochromatic tournament $T_2$.  In the former case $G$ must have only edges and in the latter only arcs.  Moreover, the underlying graph of $G$ must be bipartite as a $2$-colouring of $G$ induces a $2$-colouring of the underlying graph.

In this section we focus on the case where $G$ has only edges and is bipartite.  For ease of notation, and to align with the existing literature, we will refer to $G$ as an \emph{$m$-edge coloured graph}.  Recall we use $[m]$ as the set of edge colours, and in this case we may restrict $\Gamma$ to be a subgroup of  $S_m$.
We let $H$ be the $m$-edge coloured $K_2$ with its single edge of colour $i$, and denote $H$ by $K_2^{i}$.  

We begin with some key observations. Let $G$ be an $m$-edge coloured graph. If $\switchhom{G}{K_2^i}{\Gamma}$, then every colour appearing on an edge of $G$ must belong to the orbit of $i$ under $\Gamma$; otherwise, $G$ is a no instance. Therefore, we make the assumption that $\Gamma$ acts transitively on $[m]$.  Under this assumption $K_2^i$ is switch equivalent to $K_2^j$ for any $j \in [m]$.  Thus we have the following proposition.

\begin{proposition}\label{prop:maptoanyK2}
Fix $i \in [m]$. Let $G$ be a bipartite $m$-edge coloured graph.  The following are equivalent.
\begin{list}{(\arabic{enumi})}{\usecounter{enumi}}
  \item $\switchhom{G}{K_2^i}{\Gamma}$, 
  \item $\switchhom{G}{K_2^j}{\Gamma}$ for any $j \in [m]$,
  \item $G$ can be switched to be monochromatic of some colour $j$.
\end{list}
\end{proposition}

\begin{proof}
The implication $(1) \Rightarrow (2)$ follows from the fact that $\switchhom{K_2^i}{K_2^j}{\Gamma}$ for any $j\in [m]$ by the transitive action of $\Gamma$.  The implication $(2) \Rightarrow (3)$ is trivial.  Suppose $G$ can be switched to be monochromatic of some colour $j$.  Let $G$ have the bipartition $X \cup Y$. Since $\Gamma$ is transitive, there is $\pi \in \Gamma$ such that $\pi(j) = i$.  Then $G^{(X,\pi)}$ is monochromatic of colour $i$ implying $\switchhom{G}{K_2^i}{\Gamma}$.
\end{proof}

We have reduced the problem of determining whether an $m$-edge coloured graph $G$ is $2$-colourable to testing if $G$ is bipartite and can be switched to be monochromatic of some colour $j$.
	
In the case of signed graphs (two edge colours), $G$ can be switched to be monochromatic of colour $j$ if and only if each cycle of $G$ can be switched to be a monochromatic cycle of colour $j$~\citep{Zaslavsky_1982}.
We shall show the same result holds for bipartite $m$-edge coloured graphs. However, for our setting the question of when a cycle can be switched to be monochromatic is more complicated.  Hence, we begin by characterizing when an $m$-edge coloured even cycle can be made monochromatic.
To this end, let $G$ be a $m$-edge coloured cycle of length $2k$ on vertices $v_0, v_1, \dots, v_{2k-1}, v_0$.  By switching at $v_1$, the edge $v_0v_1$ can be made colour $i$.  Next by switching at $v_2$, the edge $v_1v_2$ can be made colour $i$.  Continuing, we see that $G$ can be switched so that all edges except $v_{2k-1}v_0$ are colour $i$. For $i,j \in [m]$, we say the cycle $G$ is \emph{nearly monochromatic of colours $(i,j)$} if $G$ has $2k-1$ edges of colour $i$ and $1$ edge of colour $j$.   Thus the problem of determining if an even cycle can be switched to be monochromatic is reduced to the problem of determining if a nearly monochromatic cycle of length $2k$ can be switched to be monochromatic.	

Let $G$ be a cycle of length $2k$ that is nearly monochromatic of colours $(i,j)$. We define a relation on $[m]$ by $j \sim_{2k} i$ if $G$ is $\Gamma$-switch equivalent to a monochromatic $C_{2k}$ of colour $i$ or equivalently $\switchhom{G}{K_2^i}{\Gamma}$.

As the definition suggests, the relation is an equivalence relation.

	\begin{lemma}
		The relation $\sim_{2k}$ is an equivalence relation.
	\end{lemma}

	\begin{proof}
The relation is trivially reflexive.
	
To see $\sim_{2k}$ is symmetric, assume $j \sim_{2k} i$. Let $G$ be a cycle of length $2k$ that is nearly monochromatic of colour $(j,i)$.  Label the vertices of the cycle in the natural order as $v_0, v_1, \dots, v_{2k-1},v_0$ where $v_0v_{2k-1}$ is the unique edge of colour $i$.  Suppose $\pi(j)=i$. Let $\Sigma = (v_1, \pi), (v_3, \pi), \dots, (v_{2k-3}, \pi)$. Then $G^\Sigma$ is nearly monochromatic of colour $(i,j)$, with edge $v_{2k-2}v_{2k-1}$ being the unique edge of colour $j$.  By assumption there is a sequence of switches, say $\Sigma'$, so that $G^{\Sigma \Sigma'}$ is monochromatic of colour $i$, giving $\switchhom{G}{K_2^i}{\Gamma}$.  Thus, $\switchhom{G}{K_2^j}{\Gamma}$ by Proposition~\ref{prop:maptoanyK2}.  That is, $G$ can be made monochromatic of colour $j$ or $i \sim_{2k} j$.
		
To prove $\sim_{2k}$ is transitive, suppose $i \sim_{2k} j$ and $j \sim_{2k} l$. Let $G, G',$ and $G''$ be $m$-edge coloured cycles of length $2k$ each with the vertices $v_0, v_1, \dots, v_{2k-1}$.	(Technically, we are considering three distinct edge colourings of the same underlying graph.)
Suppose $G, G',$ and $G''$ are nearly monochromatic of colours $(j,i)$, $(l,j)$, and $(l, i)$ respectively.  
There are $2k-1$ edges of colour $j$ in $G$ with edge $v_0v_{2k-1}$ of colour $i$ in $G$. Similarly there are $2k-1$ edges of colour $l$ in $G'$ with edge $v_0v_{2k-1}$ of colour $j$ in $G'$ and $2k-1$ edges of colour $l$ with edge $v_0v_{2k-1}$ of colour $i$ in $G''$.  We shall show $G''$ can be switched to be monochromatic of colour $l$.
		
		By hypothesis, there is a sequence $\Sigma'$ such that $G'^{\Sigma'}$ is monochromatic of colour $l$. In particular, under $\Sigma'$ all edges of colour $l$ remain colour $l$, and the edge $v_0v_{2k-1}$ changes from $j$ to $l$.  Thus, if we apply $\Sigma'$ to $G''$ the edges of colour $l$ remain colour $l$ and the product of those switches at $v_0$ and $v_{2k-1}$ changes $v_0v_{2k-1}$ from colour $i$ to colour $\sigma(i)$ for some $\sigma \in \Gamma$. We observe by the fact that $G'^{\Sigma'}$ is monochromatic, $\sigma(j) = l$. 

	We now construct a modified inverse of $\Sigma'$. Let $\Sigma''$ be the subsequence of $\Sigma'$ consisting of the switches only at $v_0$ or $v_{2k-1}$.  That is, $\Sigma''$ is a subsequence $(v_{s_0}, \pi_0), (v_{s_1}, \pi_1), \dots, (v_{s_t},\pi_t)$ where each $v_{s_r} \in \{ v_0, v_{2k-1} \}$. 
	Let $X$ (respectively $Y$) be the set of vertices of $G''$ with even (respectively odd) subscripts.  Starting with $G''^{\Sigma'}$ apply the following sequence of switches.  For $r = t, t-1, \dots, 0$, if $v_{s_r} = v_0$, then apply the switch $(X, \pi_{r}^{-1})$; otherwise, $v_{s_r}=v_{2k-1}$ and apply the switch $(Y,\pi_{r}^{-1})$.
The net effect is to apply $\sigma^{-1}$ to each edge of $G''^{\Sigma'}$. Thus each edge of colour $l$ switches to $j$ and the edge $v_0v_{2k-1}$ of colour $\sigma(i)$ becomes colour $i$. That is, we can switch $G''$ to be $G$. By hypothesis, $G$ can be switched to be monochromatic of colour $j$. By Proposition~\ref{prop:maptoanyK2}, the resulting $m$-edge coloured graph can be switched to be monochromatic of colour $l$, i.e., $i \sim_{2k} l$, as required.
	\end{proof}

We denote the equivalence classes with respect to $\sim_{2k}$ by $\agree{i}^{2k} = \{j | j \sim_{2k} i\}$. 
We now show that these classes are independent of cycle length (for even length cycles).

\begin{figure}

\begin{center}
\begin{tabular}{ccc}
\includegraphics[scale=0.55]{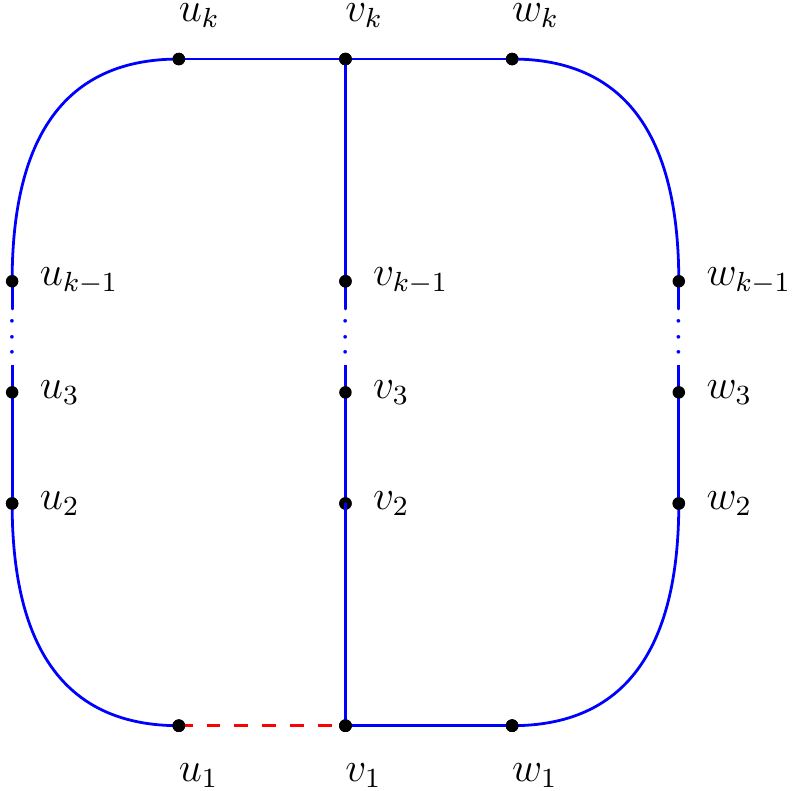} &
\includegraphics[scale=0.55]{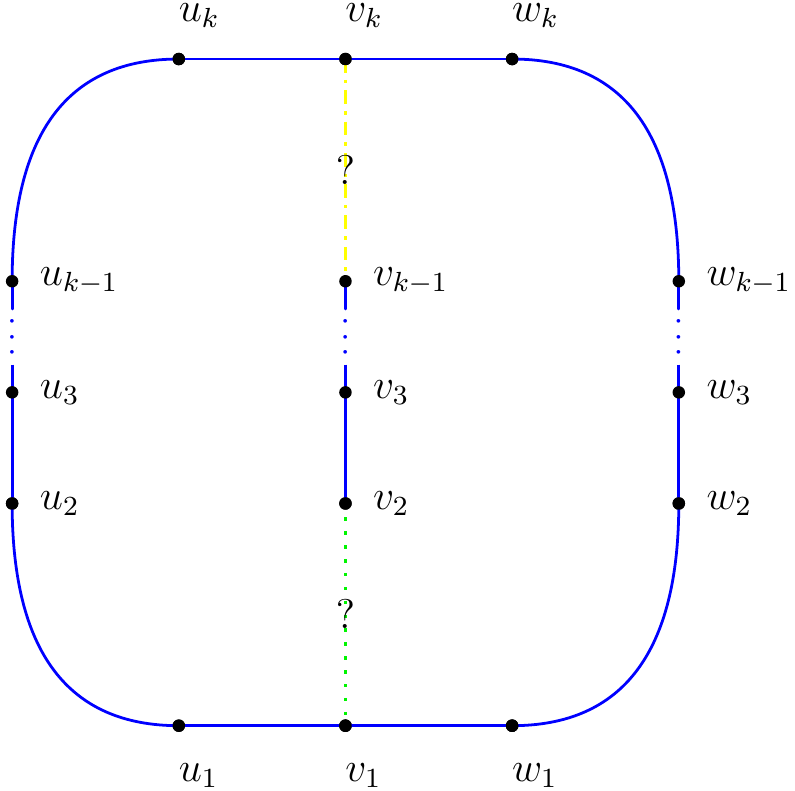} & 
\includegraphics[scale=0.55]{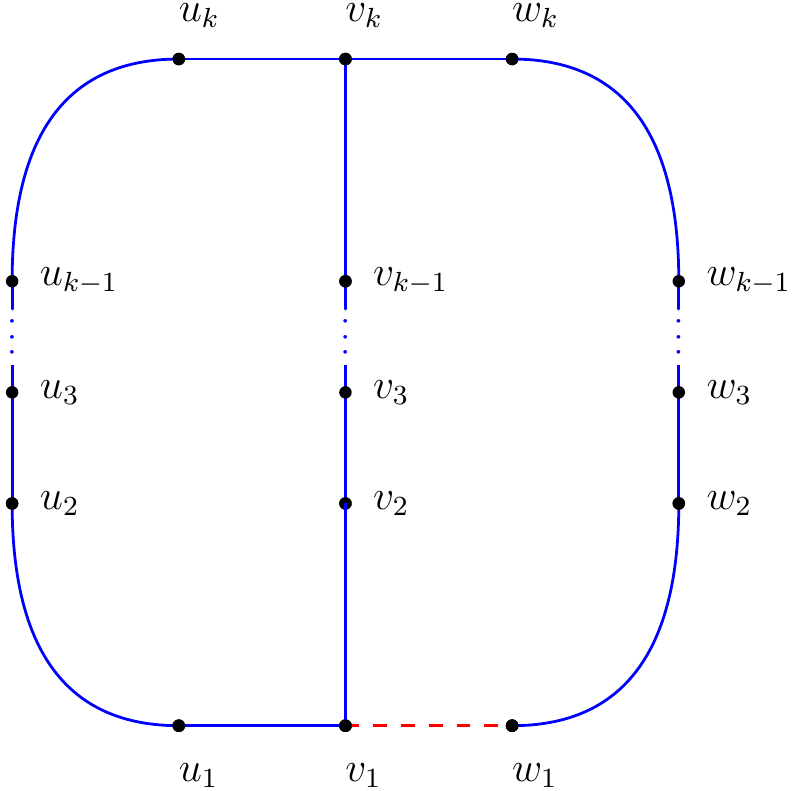} \\
$G$ & $G^{\Sigma}$ & $G^{\Sigma\Sigma'}$
\end{tabular}
\end{center}

\caption{Switching of the theta graph in Theorem~\ref{thm:4eq2k}.  Solid blue edges are colour $i$ and dashed red edges are colour $j$.}\label{fig:theta}
\end{figure}

\begin{theorem}\label{thm:4eq2k}
Let $\Gamma \leq S_m$ and $i \in [m]$.  Then $\agree{i}^{2l} = \agree{i}^{2k}$ for all $l, k \in \{2,3, \dots \}$. 
\end{theorem}
	
	\begin{proof}
Let $i \in [m]$ and let $k$ be an integer $k \geq 2$.
We show $\agree{i}^4 = \agree{i}^{2k}$ from which the result follows.

Suppose $j \in \agree{i}^4$.  Let $G$ be a cycle of length $2k$ and $H$ a cycle of length $4$ where both are nearly monochromatic of colours $(i, j)$.  Since $G \to H$ and by hypothesis, $\switchhom{H}{K_2^i}{\Gamma}$, we have $\switchhom{G}{K_2^i}{\Gamma}$ and thus $j \in \agree{i}^{2k}$.

Conversely, suppose $j \in \agree{i}^{2k+2}$.  We will show $j \in \agree{i}^{2k}$ from which we can conclude by induction that $j \in \agree{i}^4$.
Let $G$ be the $m$-edge coloured graph constructed as follows.  Let $v_1, v_2, \dots, v_k$; $u_1, u_2, \dots, u_k$; and $w_1, w_2, \dots, w_k$ be three disjoint paths of length $k-1$.  Join $v_1$ to both $u_1$ and $w_1$, and $v_k$ to both $u_k$ and $w_k$. Each edge is colour $i$ with the exception of $v_1u_1$ which is colour $j$. (Thus, $G$ is the $\theta$-graph with path lengths $k+1, k-1, k+1$.)
Denote the cycles $u_1, \dots, u_k, v_k, \dots, v_1, u_1$ and $w_1, \dots, w_k, v_k, \dots, v_1, w_1$ by $C_1$ and $C_2$ respectively.  Observe both have length $2k$, $C_1$ is nearly monochromatic of colours $(i,j)$ and $C_2$ is monochromatic of colour $i$.  Finally, let $C_3$ be the cycle $u_1, \dots, u_k, v_k, w_k, \dots, w_1, v_1, u_1$.  The cycle $C_3$ has length $2k+2$ and is nearly monochromatic of colours $(i,j)$.  See Figure~\ref{fig:theta}.
		
By assumption there exists a sequence of switches $\Sigma$ (acting on the vertices of $C_3$) such that in $G^\Sigma$ the cycle $C_3$ is monochromatic of colour $i$. We note that $v_1v_2$ and $v_{k-1}v_k$ might not be of colour $i$ in $G^\Sigma$. 
		
		There is an automorphism $\varphi$ of the underlying graph $G$ that fixes each $v_l$, $l = 1, 2, \dots, k$, and interchanges each $u_l$ with $w_l$.  We apply $\Sigma^{-1}$ to $\varphi(G^\Sigma)$ as follows.  Let $\Sigma'$ be the sequence obtained from $\Sigma$ by reversing the order of the sequence, replacing each permutation with its inverse permutation and replacing all switches on vertices $u_l$ with switches on $w_l$ and vice versa. (Switches on $v_1$ and $v_k$ are applied to $v_1$ and $v_k$ respectively.) Then in $G^{\Sigma\Sigma'}$ we see that $C_1$ is monochromatic of colour $i$.
		Therefore $\agree{i}^{2k} \supseteq \agree{i}^{2k+2}$ for all $k\geq 2$. We conclude $\agree{i}^4 = \agree{i}^{2k}$ for all $k \geq 2$.
	\end{proof}
	
As the equivalence classes depend only on the group and not the length of the cycle, we henceforth denote these classes as $\agree{i}$. 
If $j \in \agree{i}$, we say $i$ can be \emph{$\Gamma$-substituted} for $j$; that is, the single edge of colour $j$ in the cycle can be switched to colour $i$.  We call $\agree{i}$ the \emph{$\Gamma$-substitution class} for $i$.

For a fixed $m$ and $\Gamma$, $\agree{i}$ can be computed in constant time as there is a constant number of $m$-edge coloured $4$-cycles, and a constant number of (single) switches that can be applied to these cycles, from which the equivalence classes can be computed using the transitive closure.
	
	\begin{theorem}
Let $G$ be an $m$-edge coloured $C_{2k}$. It can be determined in polynomial time whether there is a $\Gamma$-switchable homomorphism of $G$ to $K_2^i$.
	\end{theorem}
	
	\begin{proof}
As described above, we can switch $G$ to be nearly monochromatic of colours $(i,j)$, for some $j$. Then $\switchhom{G}{K_2^i}{\Gamma}$ if and only if $j \in \agree{i}$. Testing this condition can be done in constant time.
	\end{proof}

We now show the \textsc{$\Gamma$-Hom-$K_2^i$} problem is polynomial time solvable.  This is accomplished by showing the problem of determining whether a given $m$-edge coloured bipartite graph can be made monochromatic of colour $i$ is polynomial time solvable.
		
We begin with the following observation that trees can always be made monochromatic.
	
\begin{lemma}\label{lem:tree}
Let $T$ be a $m$-edge coloured tree, then for any $\Gamma$, $\switchhom{T}{K_2^i}{\Gamma}$.
\end{lemma}
	
\begin{proof}
Let $T$ be a $m$-edge coloured tree. Let $v_1,v_2, \ldots, v_{|T|}$ be an ordering of $V(T)$ rooted at $v_1$ where each parent appears before all of its children (for example DFS or BFS will work). For each $k = 2,\ldots,|T|$, switch at $v_k$ so that the edge from $v_k$ to its parent in the depth first search ordering has colour $i$.
We observe that if the subtree $T[v_1,\ldots, v_{k-1}]$ is monochromatic of colour $i$, then after switching at $v_k$, so is the subtree $T[v_1,\ldots, v_k]$.
\end{proof}

Let $G$ and $H$ be $m$-edge coloured graphs such that $H$ is a subgraph of $G$. A \emph{retraction} from $G$ to $H$, is a homomorphism $r:G\to H$ such that $r(x)=x$ for all $x\in V(H)$. 
We use the following result of Hell.
	
\begin{theorem}[\cite{Hell_1972}]\label{thm:hell}
Let $G$ be a bipartite graph. Suppose $P$ is a shortest path from $u$ to $v$ in $G$. Then $G$ admits a retraction to $P$.
\end{theorem}

We now show, for general $m$-edge coloured graphs $G$, testing if $\switchhom{G}{K_2^i}{\Gamma}$ comes down to testing if each cycle admits a $\Gamma$-switchable homomorphism to $K_2^i$.  To this end define $\mathcal{C}(G)$ to be the set of cycles in an $m$-edge coloured graph $G$, and $\mathcal{F}_\Gamma$ to be the collection of cycles $C$ such that $\nswitchhom{C}{K_2^i}{\Gamma}$.

	\begin{theorem}\label{thm:tfae}
	Let $G$ be a connected $m$-edge coloured graph and $\Gamma$ a transitive group acting on $[m]$.  Suppose $i \in [m]$. The following are equivalent.
	\begin{list}{(\arabic{enumi})}{\usecounter{enumi}}
			\item $\switchhom{G}{K_2^i}{\Gamma}$.
			\item For all cycles $C \in \mathcal{C}(G)$, $\switchhom{C}{K_2^i}{\Gamma}$.
			\item $G$ is bipartite and for any spanning tree $T$ of $G$, there is a switching sequence $\Sigma$ such that in $G^\Sigma$, $T$ is monochromatic of colour $i$ and for each cotree edge the colour $i$ can be $\Gamma$-substituted for the colour of the cotree edge.
			\item For all cycles $C\in\mathcal{F}_\Gamma$, $\nswitchhom{C}{G}{\Gamma}$
		\end{list}
	\end{theorem}
	
\begin{proof}
We first prove the equivalence of statements (1), (2), and (3).

(1) $\Rightarrow$ (2) is trivially true.

(2) $\Rightarrow$ (3).  We first observe that $G$ must be bipartite as all cycles in the underlying graph map to $K_2$.  Let $T$ be a spanning tree in $G$ and let $\Sigma$ be the switching sequence constructed as in the proof of Lemma~\ref{lem:tree}.  Then $T$ is monochromatic of colour $i$ in $G^\Sigma$.  Let $e$ be a cotree edge of colour $j$.  The fundamental cycle $C_e$ in $T+e$ is nearly monochromatic of colours $(i,j)$.  By hypothesis $\switchhom{C}{K_2^i}{\Gamma}$.  Hence, $i$ $\Gamma$-substitutes for $j$.

(3) $\Rightarrow$ (1).  As above, let $T$ be a spanning tree that is monochromatic of colour $i$ in $G^\Sigma$.  Let $e_1, e_2, \dots, e_k$ be an enumeration of the cotree edges of $T$.  By hypothesis for each cotree edge $e_t$, its colour, say $j$ (in $G^\Sigma$), belongs to $\agree{i}$. 

Let $T+\{ e_1, \dots, e_t \}$ be the subgraph of $G^\Sigma$ induced by the edges of $E(T) \cup \{ e_1, \dots, e_t \}$.  By Lemma~\ref{lem:tree}, $\switchhom{T}{K_2^i}{\Gamma}$.  Suppose $\switchhom{T+\{e_1, \dots, e_{t-1} \} }{K_2^i}{\Gamma}$.  Let $e_t = uv$ have colour $j$.  Let $P$ be a shortest path from $u$ to $v$ in $T+\{e_1, \dots, e_{t-1} \}$.  By~\citep{Hell_1972}, there is a retraction $r: T+\{e_1, \dots, e_{t-1} \}  \to P$ with $r(u) = u$ and $r(v) = v$.  Adding the edge $e_t$ shows $\switchhom{T+\{e_1, \dots, e_{t} \} }{P+e_t}{\Gamma}$ where $P+e_t$ is a nearly monochromatic cycle of colours $(i,j)$.  By assumption $i$ $\Gamma$-substitutes for $j$, so $\switchhom{P+e_t}{K_2^i}{\Gamma}$ and by composition $\switchhom{T+\{e_1, \dots, e_t\}}{K_2^i}{\Gamma}$.  By induction, $\switchhom{G}{K_2^i}{\Gamma}$.

Finally, we show (1) and (4) are equivalent.  If there is $C \in \mathcal{F}_\Gamma$ such that $\switchhom{C}{G}{\Gamma}$, then $\nswitchhom{G}{K_2^i}{\Gamma}$.  Conversely, if $\nswitchhom{G}{K_2^i}{\Gamma}$, then by (2), there is a cycle $C$ in $G$ such that $\nswitchhom{C}{K_2^i}{\Gamma}$.  In particular, $C \in \mathcal{F}_{\Gamma}$ and the inclusion map gives $\switchhom{C}{G}{\Gamma}$.
\end{proof}	

Given an $m$-edge coloured graph $G$, it is easy to test condition (3) for each component. Checking $G$ is bipartite and the switching of a spanning forest can be done in linear time in $|E(G)|$. The look up for each cotree edge requires constant time.  

However, the theorem actually gives us a certifying algorithm which we now outline (under the assumption $G$ is connected). First test if $G$ is bipartite.  If it is not, then we discover an odd cycle certifying a no instance.  Otherwise construct a spanning tree, and switch so that the tree is monochromatic of colour $i$.  Either the colour of each cotree edge belongs to $\agree{i}$ or we discover a cotree edge that does not.   In the latter case we have a cycle of $C \in \mathcal{F}_\Gamma$ that certifies $G$ is a no instance.  

Thus assume all cotree edges have colours in $\agree{i}$. The proof of Theorem~\ref{thm:tfae} provides an algorithm for switching $G$ to be monochromatic of colour $i$ through lifting the switching of the retract $P+e_t$ to all of $G$.  We show how using a similar idea with $C_4$ also works and gives a clearer bound on the running time.  Let $j$ be the colour of a cotree edge, say $uv$.  Recall $j \in \agree{i}^4$.  Let $H$ be a $C_4$ with vertices labelled as  $v_0, v_1, v_2, v_3$ and edges coloured as $v_0 v_3$ is colour $j$ and all other edges are colour $i$.  Let $\Sigma$ be a switching sequence so that $H^\Sigma$ is monochromatic of colour $i$. Let $X$ (respectively $Y$) be the set of vertices of $G$ in the same part of the bipartition as $u$ (respectively $v$).  For each $(v_i, \pi_i)$ in $\Sigma$ we apply the same switch $\pi_i$ in $G$ at $u$ if $v_i = v_0$; at $X \backslash \{ u \}$ if $v_i = v_2$; at $v$ if $v_i = v_3$; and at $Y \backslash \{v \}$ if $v_i = v_1$. At the end of applying all switches in $\Sigma$, edges in $G$ that were of colour $i$ remain colour $i$, and the cotree edge $uv$ switches from $j$ to $i$.  As $|\Sigma|$ is constant (in $|\Gamma|$), this switching sequence for $uv$ requires $O(|V(G)|)$ switches. In this manner the concatenation of $|E(G)|-|V(G)|+1$ such switching sequences (together with the switches required to make $T$ monochromatic) switch $G$ to be monochromatic of colour $i$.  This sequence together with the bipartition of $G$ certifies that $\switchhom{G}{K_2^i}{\Gamma}$.  We have the following.

	\begin{corollary}\label{cor:medge2col}
	The problem \textsc{$\Gamma$-Hom-$K_2^i$} is polynomial time solvable by a certifying algorithm.
	\end{corollary}
		
\section{General $(m,n)$-coloured graphs}

In this section we show the \textsc{$\Gamma$-Switchable $k$-Col} problem is polynomial time solvable.  As noted above, a general $(m,n)$-mixed graph $G$  is $2$-colourable if it only has edges and for some edge colour $i$, $\switchhom{G}{K_2^i}{\Gamma}$ or it only has arcs and for some arc colour $i$, $\switchhom{G}{T_2^i}{\Gamma}$.  Having established the \textsc{$\Gamma$-Hom-$K_2^i$} problem is polynomial time solvable, we now show \textsc{$\Gamma$-Hom-$T_2^i$} polynomially reduces to \textsc{$\Gamma$-Hom-$K_2^i$}.  This establishes the polynomial time result of Theorem~\ref{thm:main} which we restate.

\begin{theorem}
The \textsc{$\Gamma$-Switchable $2$-Col} problem is polynomial time solvable.
\end{theorem}

\begin{proof}
Let $G$ be an instance of \textsc{$\Gamma$-Switchable $2$-Col}, i.e., an $(m,n)$-mixed graph.  If $G$ is not bipartite, we can answer No.  If $G$ has both edges and arcs, then we can answer No.  If $G$ only has edges, then by Corollary~\ref{cor:medge2col} we can choose any edge colour $i$ (we still assume $\Gamma$ is transitive) and test $\switchhom{G}{K_2^i}{\Gamma}$ in polynomial time.

Thus assume $G$ is bipartite with bipartition $(A,B)$ and has only arcs.  Analogous to Section~\ref{sec:medge}, we can view $\Gamma$ as acting transitively on the $n$ arc colours. If $\Gamma$ does not allow any arc colours to switch direction, i.e., for all $\pi \in \Gamma$, $\gamma_i(uv) = uv$ for all $i$, then $G$ must have all its arcs from say $A$ to $B$; otherwise, we can say No.  At this point $G$ may be viewed as an $n$-\emph{edge} coloured graph.  (We can ignore the fixed arc directions.) We apply the results of Section~\ref{sec:medge}.

Finally, we may assume $G$ is bipartite, with only arcs, and $\Gamma$ acts transitively on arc colours and directions.  That is, for any arc colours $i$ and $j$, $\Gamma$ contains a permutation $\pi_1$ (respectively $\pi_2$) that takes an arc $uv$ of colour $i$ to an arc $uv$ (respectively $vu$) of colour $j$.

We now construct a $(2n)$-\emph{edge} coloured graph $G'$ as follows. Let $V(G') = V(G)$. If there is an arc of colour $i$ from $u \in A$ to $v \in B$, we put an edge $uv$ of colour $i^+$ in $G'$, and if there is an arc of colour $i$ from $v \in B$ to $u \in A$, we put an edge $uv$ of colour $i^-$ in $G'$.

From $\Gamma$ we construct a new group $\Gamma' \leq S_{2n}$.  Note that $\Gamma$ as described above acts on $(m,n)$-mixed graphs and $\Gamma'$ will be naturally restricted to act on $(2n)$-edge coloured graphs.  Let $\pi = (\alpha, \beta, \gamma_1, \dots, \gamma_n) \in \Gamma$.  Define $\pi' \in \Gamma'$ as follows. For each arc colour $i$,

$$
\pi'(i^+) = \left\{ \begin{matrix} 
  \beta(i)^+ & \mbox{ if } \gamma_i(uv) = uv \\
  \beta(i)^- & \mbox{ if } \gamma_i(uv) = vu
  \end{matrix} \right.
  \hspace{1em}
  \mbox{ and }
  \hspace{1em}
 \pi'(i^-) = \left\{ \begin{matrix} 
  \beta(i)^- & \mbox{ if } \gamma_i(uv) = uv \\
  \beta(i)^+& \mbox{ if } \gamma_i(uv) = vu
  \end{matrix} \right. 
$$

It can be verified that the mapping $\pi \to \pi'$ is a group isomorphism.

The translation of $G$ to $G'$ can be expressed as a function $F(G) = G'$.  It is straightforward to verify $F$ is a bijection from $n$-arc coloured graphs to $2n$-edge coloured graphs provided we fix the bipartition $V(G) = A \cup B$.  Moreover, if $\pi \in \Gamma$ and $\pi'$ is the resulting permutation in $\Gamma'$, then again it is easy to verify that $F(G^{(v,\pi)}) = (G')^{(v,\pi')}$ for any $v$ in $V(G) = V(G')$. 

Suppose $\switchhom{G}{T_2^i}{\Gamma}$.  By the transitivity of $\Gamma$, we may assume that $T_2^i$ has its tail in $A$, and thus all arcs in $G$ can be switched to be colour $i$ with their tail in $A$.  The corresponding switches on $G'$ switch all edges to colour $i^+$.  That is, $\switchhom{G'}{K_2^{i^+}}{\Gamma'}$.  On the other hand, if $\switchhom{G'}{K_2^{i^+}}{\Gamma'}$, then the corresponding switches on $G$ show that $\switchhom{G}{T_2^i}{\Gamma}$ (with the vertices of $A$ mapping to the tail of $T_2^i$).
\end{proof}

We conclude this section with a remark on the number of switches required to change the input $G$ to be monochromatic.  There are $|V(G)|-1$ switches required to change a spanning tree of $G$ to be monochromatic of colour $i$.  To change the cotree edges to colour $i$ (assuming each is of a colour in $\agree{i}$), we claim at most $c_\Gamma|V(G)|$ switches are required where $c_\Gamma$ is a constant depending on $\Gamma$ and the number of colours ($m$ and $n$).  We argue only for $m$-edge coloured graphs, given the reduction above.  For (a labelled) $C_4$, there are $m^4$ edge colourings.  For each vertex there are $|\Gamma|$ switches.  The \emph{reconfiguration graph} $\mathcal{C}$ has a vertex for each edge-colouring of $C_4$ and an edge joining two vertices if there is a single switch that changes one into the other.  (The existence of inverses ensures this is an undirected graph.)  Thus, $\mathcal{C}$ has order $m^4$ and is regular of degree $|\Gamma|$.  Given $j \in \agree{i}$, there is a path in $\mathcal{C}$ from a nearly monochromatic $C_4$ of colours $(i,j)$ to a monochromatic $C_4$ of colour $i$.  The switches on this path can be lifted to $G$ so that the spanning tree remains of colour $i$ and the cotree edge switches to colour $i$.  The total number of switches is at most $\max \{\mathrm{diam}(\mathcal{C}') \} \cdot |V(G)|$ where $\mathcal{C}'$ runs over all components of $\mathcal{C}$.  Thus we have the following.

\begin{proposition}\label{prop:recon}
Let $G$ be a $m$-edge coloured bipartite graph. Let $\Gamma$ be a group acting transitively on $[m]$. If $G$ is $\Gamma$-switch equivalent to a monochromatic graph, then the sequence $\Sigma$ of switches which transforms $G$ to be monochromatic satisfies,
$$
|\Sigma| \leq |V(G)| - 1 + c_\Gamma|V(G)|(|E(G)|-|V(G)|+1)
$$
where $c_\Gamma$ depends only on $\Gamma$ and $m$.
\end{proposition}

In the case that $\Gamma$ is Abelian, the switches in $\Sigma$ can be reordered, then combined, so that each vertex is switched only once.  
	
	\section{Conclusion}

We have established a dichotomy for 	the \textsc{$\Gamma$-Switchable $k$-Col} problem.  This is a step in obtaining a dichotomy theorem for \textsc{$\Gamma$-Hom-$H$} for all $(m,n)$-mixed graphs $H$ and all transitive permutation groups $\Gamma$.  Work towards a general dichotomy is the focus of our companion paper~\citep{Brewster_KM_2022}.

If $\Gamma$ and $m$ are part of the input to \textsc{$\Gamma$-Switchable $k$-Col}, then the analysis above with constant time look-ups for the equivalence classes no longer holds.  A na\"{i}ve computation of the closure of the reconfiguration graph is not possible in polynomial time, but the required computation of the equivalence classes $\agree{i}$ is possible without considering all switches.  This is studied further in~\citep{Brewster_KM_2022}.

Finally, we have restricted our attention to simple $m$-edge coloured graphs in this work.  A natural next step is to allow parallel edges (of different colours) and loops.  We believe with parallel edges and loops the analysis of \textsc{$\Gamma$-Hom-$H$} will be interesting, but non-trivial, similar to the situation for signed graphs~\citep{Brewster_FHN_2017, Brewster_Siggers_2018}.

\acknowledgements
We thank the referees for their helpful suggestions.  
	
	\bibliography{GammaSwitchBib}{}
\bibliographystyle{abbrvnat}

\end{document}